\documentclass{amsart}
\usepackage{amssymb,amstext,amsmath,amscd,amsthm,amsfonts,enumerate,latexsym, comment}
\usepackage{xcolor}
\usepackage[all]{xy}

\usepackage{hyperref}

\theoremstyle{plain}
\newtheorem{thm}{Theorem}[section]

\newtheorem*{thm*}{Theorem}
\newtheorem*{cor*}{Corollary}
\newtheorem{prop}[thm]{Proposition}

\newtheorem{lem}[thm]{Lemma}
\newtheorem{cor}[thm]{Corollary}

\newtheorem{claim}{Claim}
\newtheorem*{claim*}{Claim}

\theoremstyle{definition}
\newtheorem{defn}[thm]{Definition}

\newtheorem{ex}[thm]{Example}

\newtheorem{rem}[thm]{Remark}

\newtheorem{setup}{Setup}

\theoremstyle{remark}

\numberwithin{equation}{thm}

\newcommand{\rmQ}{\mathrm{Q}}

\newcommand{\calX}{\mathcal{X}}

\newcommand{\fkm}{\mathfrak{m}}

\newcommand{\fkp}{\mathfrak{p}}

\def\Hom{\operatorname{Hom}}
\def\Im{\mathrm{Im}}

\def\Spec{\operatorname{Spec}}

\def\tr{\operatorname{tr}}

\def\Aut{\operatorname{Aut}}

\title[]{Traces of semi-invariants}

\author{Ela Celikbas}
\address{E. Celikbas: School of Mathematical and Data Sciences, West Virginia University, Morgantown, WV 26506, USA.}
\email{ela.celikbas@math.wvu.edu}

\author{J\"{u}rgen Herzog}
\address{J. Herzog: Fachbereich Mathematik, Universit\"at Duisburg-Essen, Fakult\"at f\"ur Mathematik, 45117 Essen, Germany}
\email{juergen.herzog@uni-essen.de}

\author{Shinya Kumashiro}
\address{S. Kumashiro: National Institute of Technology, Oyama College, 771 Nakakuki, Oyama, Tochigi, 323-0806, Japan}
\email{skumashiro@oyama-ct.ac.jp}

\thanks{S. Kumashiro was supported by JSPS KAKENHI Grant Number JP21K13766 and by Grant for Basic Science Research Projects from the Sumitomo Foundation (Grant number 2200259).}

\begin{document}

\begin{abstract}
This article investigates the traces of certain modules over rings of invariants associated with finite groups. More precisely, we provide a formula for computing the traces of arbitrary semi-invariants, thereby contributing to the understanding of the non-Gorenstein locus of rings of invariants. Additionally, we discuss applications of this formula, including criteria for rings of invariants to be Gorenstein on the punctured spectrum and nearly Gorenstein, as well as criteria for semi-invariants to be locally free.
%Using this formula, we establish criteria for rings of invariants to be Gorenstein on the punctured spectrum and nearly Gorenstein.}
\end{abstract}

\maketitle

%%%%%%%%%%%%%%%%%%%%%%%%%%%%%%%%%%%%%%%%%%%%%%%%%%%%%%%%%%%%
\section{Introduction}\label{section1}

The purpose of this article is to explore the traces of certain modules over rings of invariants. Let $R$ be a commutative Noetherian ring, and  let $M$ be a finitely generated $R$-module. The {\bf trace} of $M$, denoted as $\tr_R(M)$, is defined by 
\[
\tr_R(M):=\sum_{f\in\Hom(M, R)} f(M).
\]

The significance of studying traces of modules becomes evident through a straightforward observation: $\tr_R(M) =R$ if and only if there exists $n>0$ such that $M^n$ has a free summand. In the local case, $n=1$ suffices. Thus, the behavior of the trace of a module is closely linked to the decomposition of the module. Numerous studies leverage this observation, including the classification of indecomposable maximal Cohen-Macaulay modules over one-dimensional Cohen-Macaulay local rings of multiplicity 
2 (\cite[Section 7]{Bass}, also see \cite[Theorem 1.1]{IK}), and the investigation of the closedness of the non-Gorenstein locus of $R$ (\cite[p.199, before Theorem 11.42]{LW}). 
%When $R$ is local, one can choose $1$ as $n$.  

Moving forward,  we survey rings of invariants. For a subgroup $G$ of the automorphism group $\Aut(R)$, the subring of $R$, denoted as $R^G$, is defined as
\[
R^G:=\{a\in R \mid \sigma(a) = a \text{ for all $\sigma \in G$}\},
\]
and is known as the {\bf ring of invariants}. The study of the invariant theory of finite groups is a classical subject that cannot be comprehensively covered in this article. For further information, one can consult the sources such as \cite{Ben, Stan}. Among the studies of rings of invariants, some of the most famous results include the Cohen-Macaulay property of rings of invariants and the characterization of Gorenstein invariants, as presented in the following theorem.

\begin{thm} \label{t11} 
\begin{enumerate}[\rm(1)] 
\item {\rm (\cite{HE}):} Assume $R$ is a Cohen-Macaulay ring, and $G$ is a finite group whose order is invertible in $R$. Then, $R^G$ is Cohen-Macaulay.
\item {\rm (\cite{Wat}):} Let $K$ be a field of characteristic $0$, $R=K[X_1, \dots, X_d]$, and $G$ a finite subgroup of $\mathrm{GL}(K^d)$. (We identify $\sigma =(a_{ij}) \in \mathrm{GL}(K^d)$ with the automorphism $\varphi: R \to R; X_j\mapsto \sum_{i=1}^d a_{ij}X_i$.) Consider the conditions:
\begin{enumerate}[\rm(i)] 
\item $R^G$ is Gorenstein.
\item $G\subseteq \mathrm{SL}(K^d)$. 
\end{enumerate}
Then, (ii) $\Rightarrow$ (i) holds. If $G$ has no pseudo-reflection (see Definition \ref{d34}), (i) $\Rightarrow$ (ii) holds as well.
\end{enumerate}
\end{thm}

%Ela: In the previous theorem, I added commas after ``Then" in both items for better readability, and I also included a colon after ``Consider the following conditions" and replaced ``too" with "as well" at the end.

%Now we explain the purpose of this article in detail. 

In what follows, let $K$ be a field, $R=K[X_1, \dots, X_d]$ be the polynomial ring over $K$, and $G$ be a finite subgroup of $\mathrm{GL}(K^d)$. The purpose of this article is to refine Theorem \ref{t11}(2). Specifically,  we aim to provide a method for determining the non-Gorenstein locus of $R^G$,
\[
\{ \fkp\in \Spec R^G \mid R^G_\fkp \text{ is not Gorenstein} \},
\]
by presenting a formula for computing the traces of semi-invariants over rings of invariants. Here, recall that for a group homomorphism $\calX:G \to \mathrm{GL}(K)$, 
\[
R^\calX:=\{a\in R \mid \sigma(a) = \calX(\sigma)a \text{ for all $\sigma \in G$}\} 
\]
is an $R^G$-module and is called the {\bf semi-invariants}. In general, the non-Gorenstein locus can be computed by the trace of the canonical module (see Remark \ref{f22}(4)). Additionally, the canonical module of $R^G$ is the semi-invariants of a certain group homomorphism called the inverse determinant character (see before Corollary \ref{c41}). Therefore, to compute the non-Gorenstein locus of $R^G$, it is sufficient to provide a formula for computing the traces of semi-invariants. Indeed, we give such a formula for arbitrary semi-invariants under certain assumptions. The main result of this article is captured in the following theorem.

%(Ela:  I made small changes in the first sentence and added a comma after ``nonzero" in the next theorem.)

\begin{thm} {\rm (Theorem \ref{t38})} \label{t12}
Let $K$ be an algebraically closed field, and let $G$ be a finite abelian subgroup of $\mathrm{GL}(K^d)$ generated by $\sigma_1, \dots, \sigma_\ell$. After a suitable choice of a basis for $K^d$, we may assume that each $\sigma_i$ is a diagonal matrix with diagonal entries $\xi_i^{t_{i1}}, \xi_i^{t_{i2}}, \dots, \xi_i^{t_{id}}$ for some non-negative integers $t_{ij}$ and the $n_i$th primitive root $\xi_i$ of $1\in K$. We may further assume that $\gcd(t_{i1}, \dots, t_{id}, n_i) = 1$ for all $1 \le i \le \ell. $ 
With this notation, we also assume that $n_1, \dots, n_\ell$ are pairwise coprime, and $G$ has no pseudo-reflection. 
Then, for all characters $\calX$, $R^\calX$ is nonzero, and the formula 
\[
\tr_{R^G} (R^\calX) = R^\calX R^{\calX^{-1}}
\] 
holds, where $\calX^{-1}$ denotes the inverse character of $\calX$, mapping $\sigma \in G$ to $\calX(\sigma^{-1}) \in \mathrm{GL}(K)$. 
\end{thm}

As a consequence of Theorem \ref{t12}, we obtain a criterion for  a semi-invariants to be locally free on the punctured spectrum (Theorem \ref{t311}). In particular, we derive a criterion for rings of invariants to be Gorenstein on  the punctured spectrum (Corollary \ref{c41}). We further explore the nearly Gorenstein property, which was recently introduced and studied with the aim of developing a theory for rings that are close to being Gorenstein (\cite{HHS}).

The rest of this article is organized as follows. In Section \ref{section2}, we survey fundamental properties of traces of modules, which we use throughout this article. In Section \ref{section3}, we prove Theorem \ref{t12}. In Section \ref{section4}, we apply Theorem \ref{t12} with the canonical module and  provide the criteria noted in the previous paragraph. Examples illustrating our results are also presented.

%The rest of this article is organized as follows. In Section \ref{section2} we survey fundamental properties of traces of modules, which we use throughout this article. In Section \ref{section3} we prove Theorem \ref{t12}. In Section \ref{section4} we apply Theorem \ref{t12} with the canonical module, and we give criteria noted in the previous paragraph. Examples illustrating our results are also given. 

%%%%%%%%%%%%%%%%%%%%%%%%%%%%%%%%%%%%%%%%%%%%%%%%%%%%%%%%%%%%
\section{Traces of modules}\label{section2}

Let $A$ be a commutative Noetherian ring, and let $M$ be a finitely generated $A$-module. In this section, we summarize basic properties of trace. 

\begin{defn}
\[
\tr_A(M):=\sum_{f\in\Hom(M, A)} f(M)
\]
is called the {\bf trace} of $M$.
\end{defn}

\begin{rem}\label{f22} 
\begin{enumerate}[\rm(1)] 
\item $\tr_A(M) = \Im(\mathrm{ev})$, where $\mathrm{ev}:M\otimes_A \Hom(M, A) \to A; x\otimes f \mapsto f(x)$ for $x\in M$ and $f\in \Hom(M, A)$.
\item (\cite[Lemma 1.1]{HHS}) Let $I$ be an ideal of $A$. If $I$ contains a non-zerodivisor of $A$, then $$\tr_A(I)=(A:_{\rmQ(A)}~I)I,$$
where $\rmQ(A)$ denotes the total ring of fraction of $A$. 
\item (\cite[Proposition 2.8(viii)]{Lin}) %Suppose that $A$ is a Noetherian ring and $M$ is a finitely generated $A$-module. Then, 
$S^{-1} \tr_A(M) = \tr_{S^{-1}A} S^{-1}M$ for all multiplicative closed subset $S$ of $A$.  
%$\tr_A(M)_\fkp = \tr_{A_\fkp} (M_\fkp)$ for all $\fkp \in \Spec A$. 
\item Suppose that $A$ is a Noetherian graded ring having the unique graded maximal ideal, and $M$ is a finitely generated graded $A$-module. Then, $\tr_A(M)=A$ if and only if $M$ has a $A$-free summand. Furthermore, letting ${}^*\Spec A$ be the set of graded prime ideals of $A$, we have 
\[
\{ \fkp\in {}^*\Spec A \mid \fkp\not\supseteq \tr_A(M) \} = \{ \fkp\in {}^*\Spec A \mid M_\fkp \text{ has an $A_\fkp$-free summand} \}.
\]
%Suppose that $A$ is a Noetherian graded ring having unique graded maximal ideal and $M$ is a finitely generated graded $A$-module. Then, $\tr_A(M)=A$ if and only if $M$ has a $A$-free summand. Therefore, by (3), we have 
%\[
%\{ \fkp\in \Spec A \mid \fkp\not\supseteq \tr_A(M) \} = \{ \fkp\in \Spec A \mid M_\fkp \text{ has an $A_\fkp$-free summand} \}.
%\]
\end{enumerate}
\end{rem}

\begin{proof} 
(4): The proof of the former part proceeds almost the same way as in the local case; see \cite[Proposition 2.8(iii)]{Lin}. 
We now prove the latter part. 

$(\supseteq)$: Let $\fkp\in {}^*\Spec A$ such that $M_\fkp$ has an $A_\fkp$-free summand. Then, $\tr_A(M)_\fkp = \tr_{A_\fkp} (M_\fkp)= A_\fkp$ by (3). Hence, $\fkp\not\supseteq \tr_A(M)$. 

$(\subseteq)$: Let $\fkp\in {}^*\Spec A$ such that $\fkp\not\supseteq \tr_A(M)$. Let $S$ be the set of homogeneous elements of $A$ not belonging to $\fkp$. Then $\tr_{S^{-1} A} S^{-1} M = S^{-1} \tr_A(M) = S^{-1} A$ since $\tr_A(M)$ is a graded ideal of $A$ and $\fkp\not\supseteq \tr_A(M)$. Since $S^{-1} A$ has the unique graded maximal ideal $S^{-1} \fkp$, it follows that $S^{-1} M$ has an $S^{-1} A$-free summand. By localizing at $\fkp$, we obtain that $M_\fkp$ has an $A_\fkp$-free summand. 
\end{proof}

%%%%%%%%%%%%%%%%%%%%%%%%%%%%%%%%%%%%%%%%%%%%%%%%%%%%%%%%%%%%
\section{The Trace of $R^\calX$}\label{section3}

\begin{setup}\label{setup1}
In what follows, throughout this article, let 
\begin{itemize}
\item $K$ be an algebraically closed field, %$K$ be an algebraically closed field of characteristic $0$, 
\item $R=K[X_1, X_2, \dots, X_d]$ be the polynomial ring over $K$ with $d\ge 2$ and $\deg X_j=1$ for $1 \le j \le d$, and 
\item $G=\langle \sigma_1, \dots, \sigma_\ell \rangle\subseteq \mathrm{GL}(K^{d})$ be a finite abelian group whose order is not divisible by the characteristic of $K$. After a suitable choice of a basis for $K^d$, we may assume that each $\sigma_i$ is a diagonal matrix with diagonal entries $\xi_i^{t_{i1}}, \xi_i^{t_{i2}}, \dots, \xi_i^{t_{id}}$ for some non-negative integers $t_{ij}$ and the $n_i$th primitive root $\xi_i$ of $1\in K$.
\end{itemize}
We may assume that $\gcd(t_{i1}, \dots, t_{id}, n_i) =1$ for all $1 \le i \le \ell$. 
With this notation, the graded subring
\[
R^G:=\{a\in R \mid \sigma(a) = a \text{ for all $\sigma \in G$}\}
\]
of $R$ is called the {\bf ring of invariants}. Let $\fkm_G$ be the graded maximal ideal of $R^G$. 
A group homomorphism $\calX: G \to \mathrm{GL}(K)$ is called a {\bf character} of $G$. For a character $\calX$, we denote the {\bf inverse character} of $\calX$ by $\calX^{-1}$, which maps $\sigma \in G$ to $\calX(\sigma^{-1}) \in \mathrm{GL}(K)$. 
Each character $\calX$ defines an $R^G$-module 
\[
R^\calX:=\{a\in R \mid \sigma(a) = \calX(\sigma)a \text{ for all $\sigma \in G$}\}, 
\]
and we call it the {\bf semi-invariants of weight $\calX$}. In our assumption on $G$, $\calX$ is determined by $\calX(\sigma_i)$ for all $1 \le i \le \ell$, and $\calX(\sigma_i)$ must be $\xi_i^{s_i}$ for some $1 \le s_i \le n_i$ since $\sigma_i^{n_i} = 1_G$.
For a character $\calX$ such that $\calX(\sigma_i)=\xi_i^{s_i}$ for $1 \le i \le \ell$, we denote $R^\calX$ by $R^{(s_1, \dots, s_\ell)}$ when we want to clarify the action of $\calX$. 
\end{setup}

\begin{rem} \label{rem0}
The following statements hold true.
\begin{enumerate}[\rm(1)] 
\item $\displaystyle R=\bigoplus_{\calX\text{ is a character}} R^\calX$. 
\item $R^\calX$ is a maximal Cohen-Macaulay $R^G$-module generated by monomials, provided $R^\calX \ne 0$. 
\item $R^\calX \ne 0$ if and only if $R^{\calX^{-1}} \ne 0$. 
\item $R^\calX R^{\calX^{-1}} \subseteq R^G$. 
\item $R^\calX$ is a torsion-free $R^G$-module of rank $1$, provided $R^\calX \ne 0$. 
\end{enumerate}
\end{rem}

\begin{proof}
(1): This is clear.

(2): It is straightforward to check that $R^\calX$ is an $R^G$-module generated by monomials. Set $n=\prod_{i=1}^\ell n_i$. Then, $X_1^n, \dots, X_d^n\in R^G$, and $X_1^n, \dots, X_d^n$ form a regular sequence on $R$. Thus, $R$ is a Cohen-Macaulay $R^G$-module of dimension $d$, and so is $R^\calX$ by (1). 

(3): Suppose that $R^\calX \ne 0$. By (2), we can choose a nonzero monomial $f\in R^\calX$. Choose a positive integer $s$ such that $(X_1^n\cdots X_d^n)^s/f$ is a monomial. Since $(X_1^n\cdots X_d^n)^s\in R^G$, we observe that $(X_1^n\cdots X_d^n)^s/f \in R^{\calX^{-1}}$.

(4): Let $f\in R^{\calX^{-1}}$ and $g\in R^\calX$. Then, $\sigma(fg) = \sigma(f)\sigma(g) = \calX^{-1}(\sigma)f{\cdot}\calX(\sigma)g = \calX(\sigma^{-1})f{\cdot}\calX(\sigma)g = fg$. Hence, $fg\in R^G$. 

(5): We can choose a nonzero element $f\in R^{\calX^{-1}}$ by (3). Then, $fR^{\calX} \cong R^\calX$ and $fR^{\calX}\subseteq R^G$ by (4). Since $fR^{\calX}$ is a torsion-free $R^G$-module of rank $1$, so is $R^\calX$.
\end{proof}

\begin{lem}\label{l32}
Set $S=R^G\setminus \{0\}$  as a multiplicative closed subset of $R^G$. Then, 
\[
\tr_{R^G}(R^\calX) = (R^G:_{S^{-1}R} R^\calX) R^\calX.
\]
\end{lem}

\begin{proof}
We may assume that $R^\calX$ is nonzero. Let $f\in R^{\calX^{-1}}$ be a nonzero element (see Remark \ref{rem0}(3)). Then, $fR^{\calX} \cong R^\calX$ and $fR^{\calX}\subseteq R^G$ by Remark \ref{rem0}(4). Hence, $\tr_{R^G}(R^\calX) = \tr_{R^G}(fR^\calX) = (R^G:_{\rmQ(R^G)}~fR^\calX) fR^\calX$ by Remark \ref{f22}. We have 
\[
R^G:_{\rmQ(R^G)} fR^\calX = (R^G:_{S^{-1}R} fR^\calX) \cap \rmQ(R^G) \subseteq R^G:_{S^{-1}R} fR^\calX = f^{-1} (R^G:_{S^{-1}R} R^\calX). 
\]
It follows that 
\[
\tr_{R^G}(R^\calX) \subseteq f^{-1} (R^G:_{S^{-1}R} R^\calX) fR^\calX = (R^G:_{S^{-1}R} R^\calX) R^\calX. 
\] 
On the other hand, for each $\alpha \in R^G:_{S^{-1}R} R^\calX$, we can consider an $R^G$-linear homomorphism 
\[
\hat{\alpha}:R^\calX \to R^G; h \mapsto \alpha h
\] 
for $h\in R^\calX$. Since $\Im \hat{\alpha} = \alpha R^\calX$, it follows that $(R^G:_{S^{-1}R} R^\calX) R^\calX \subseteq \tr_{R^G}(R^\calX)$. Hence, we have $\tr_{R^G}(R^\calX) = (R^G:_{S^{-1}R} R^\calX) R^\calX$. 
\end{proof}

In general, computing $R^G:_{S^{-1}R} R^\calX$ for a given semi-invariant $R^\calX$ of weight $\calX$ can be challenging. Thus, in the following, we provide a computable method for $R^G:_{S^{-1}R} R^\calX$. To state our assertion simply, let $A$ be an infinite subset of $R$, and we say that $\gcd(A) = 1$ if there exists a finite subset $B$ of $A$ such that $\gcd(B)=1$.

\begin{prop}\label{p33}
$R^G:_{S^{-1} R} R^\calX \supseteq R^{\calX^{-1}}$ holds. Moreover, if $\gcd (R^\calX) =1$, then $R^G:_{S^{-1} R} R^\calX = R^{\calX^{-1}}$. %In this case, we obtain that $\tr_{R^G}(R^\calX) = R^\calX R^{\calX^{-1}}$.
\end{prop}

\begin{proof}
The inclusion $R^G:_{S^{-1} R} R^\calX \supseteq R^{\calX^{-1}}$ follows from Remark \ref{rem0}(4). Suppose that $\gcd (R^\calX) =1$. Let $a/b\in R^G:_{S^{-1} R} R^\calX$, where $a, b\in R$. Thus, $(a/b) R^\calX \subseteq R^G$. We may assume that $\gcd(a,b)=1$ (note that we do not assume that $b\in S$). Since $a$ and $b$ are coprime, $b$ divides all elements in $R^\calX$. It follows that $b\in K\setminus \{0\}$ since $\gcd (R^\calX) =1$. Hence, $aR^\calX = ab^{-1} R^\calX \subseteq R^G$, thus $a\in R^{\calX^{-1}}$. This concludes that $a/b=ab^{-1}\in R^{\calX^{-1}}$.
\end{proof}

By Proposition \ref{p33}, it is natural to ask when $\gcd (R^\calX) =1$ holds. To consider this problem, we need the notion of pseudo-reflection. Recall that we assume that $G \subseteq \mathrm{GL}(K^{\oplus d})$.

\begin{defn} \label{d34} (\cite[before Theorem 6.4.10]{BH})
For an element $\sigma \in \mathrm{GL}(K^{\oplus d})$, $\sigma$ has a {\bf pseudo-reflection} if $\sigma$ has finite order, and its eigenspace for the eigenvalue $1$ has dimension $d-1$.
\end{defn}

\begin{lem}\label{l35} {\rm (cf. \cite[Section 2]{CS})}
For $1\le i \le \ell$, the following are equivalent. 
\begin{enumerate}[\rm(1)] 
\item A cyclic subgroup $\langle \sigma_i \rangle$ of $G$ has no pseudo-reflection.
\item $\gcd(t_{i\, j_1}, t_{i\, j_2}, \dots, t_{i\, j_{d-1}}, n_i) = 1$ for all $(d-1)$-tuples with distinct integers $j_1,\dots, j_{d-1}\in \{1,2,\dots, d\}$.  
\end{enumerate}
\end{lem}

\begin{proof}
We prove the contrapositive of the assertion. Observe that 
\begin{align*}
& \text{$\langle \sigma_i \rangle$ has a pseudo-reflection} \\
\Leftrightarrow \quad & \text{there exists $1\le s <n_i$ such that $\sigma_i^s$'s eigenspace for the eigenvalue $1$ has dimension $d-1$}\\
\Leftrightarrow \quad & \text{there exist $1\le s <n_i$ and $1\le j\le d$ such that}\\
& \text{$st_{i\, j} \not\equiv 0 \mod n_i$ \quad and \quad $st_{i1}\equiv \cdots \equiv st_{i\, j-1} \equiv  st_{i\, j+1} \equiv \cdots \equiv st_{i\, d} \equiv 0 \mod n_i$.}
\end{align*}
Thus, it is enough to prove that the last assertion above is equivalent to saying that 
\begin{align}\label{eq351}
\gcd(t_{i1}, t_{i2}, \dots, t_{i\, j-1}, t_{i\, j+1}, \dots, t_{i\, d}, n_i) \ne 1
\end{align}
for some $1 \le j \le d$. Set $g:=\gcd(t_{i1}, t_{i2}, \dots, t_{i\, j-1}, t_{i\, j+1}, \dots, t_{i\, d}, n_i)$. If \eqref{eq351} holds true, then we can choose $n_i/g$ as $s$. Indeed, since we assume that $\gcd(t_{i1}, \dots, t_{id}, n_i) =1$ (see Setup \ref{setup1}), we have $\gcd(g, t_{ij}) = 1$. Therefore, since $t_{ij}/g$ is not an integer, $st_{ij} \not\equiv 0 \mod n_i$. The assertion that $st_{i1}\equiv \cdots \equiv st_{i\, j-1} \equiv  st_{i\, j+1} \equiv \cdots \equiv st_{i\, d} \equiv 0 \mod n_i$ follows since $t_{i1}/g, \dots, t_{i\, j-1}/g, t_{i\, j+1}/g, \dots, t_{id}/g$ are integers. 

Conversely, assume that \eqref{eq351} is not true, i.e., $g = 1$. Then, 
\[
\gcd(st_{i1}, st_{i2}, \dots, st_{i\, j-1}, st_{i\, j+1}, \dots, st_{i\, d}, n_i)
\] 
divides $sg = s$. Since $1\le s <n_i$, 
\[
st_{i1}\equiv \cdots \equiv st_{i\, j-1} \equiv  st_{i\, j+1} \equiv \cdots \equiv st_{i\, d} \equiv 0 \mod n_i
\] 
does not hold true. 
\end{proof}

\begin{cor}\label{c36}
If $G$ has no pseudo-reflection, then $\gcd(t_{i\, j_1}, t_{i\, j_2}, \dots, t_{i\, j_{d-1}}, n_i) = 1$ for all $1\le i \le \ell$ and $(d-1)$-tuples with distinct integers $j_1,\dots, j_{d-1}\in \{1,2,\dots, d\}$.  
\end{cor}

\begin{proof}
Since $G$ has no pseudo-reflection, neither do all cyclic subgroups $\langle \sigma_i \rangle$ for all $i=1,\dots, \ell$. Applying Lemma \ref{l35} yields the assertion. 
\end{proof}

The following proposition is key to proving the main theorem. 

\begin{prop}\label{p37}
Let $a_{ij}, b_i, p_i$ be positive integers for $1\le i \le m$ and $1\le j \le n$. Consider the following simultaneous (congruence) equation: 
\begin{align}\label{eqstar}
\begin{cases}
a_{11} x_1 + a_{12} x_2 +\cdots +a_{1n} x_n &\equiv b_1 \mod p_1 \\
a_{21} x_1 + a_{22} x_2 +\cdots +a_{2n} x_n &\equiv b_2 \mod p_2 \\
&\vdots \\
a_{m1} x_1 + a_{m2} x_2 +\cdots +a_{mn} x_n &\equiv b_m \mod p_m
\end{cases}
\end{align}
If $p_1, \dots, p_m$ are pairwise coprime and $\gcd(a_{i1}, \dots, a_{in}, p_i) =1$ for all $1 \le i \le m$, then there exists a positive integer solution $x_1, \dots, x_n$ of \eqref{eqstar}.
\end{prop}

\begin{proof}
While this assertion is likely known in some literature, we were unable to find a direct reference. Therefore, we include a proof for completeness. 

We prove by induction on $n$. Suppose that $n=1$, that is, 
\begin{align*}
\begin{cases}
a_{11} x_1  &\equiv b_1 \mod p_1 \\
a_{21} x_1  &\equiv b_2 \mod p_2 \\
&\vdots \\
a_{m1} x_1 &\equiv b_m \mod p_m
\end{cases}
\end{align*}
For each $1 \le i \le m$, since $\gcd(a_{i1}, p_i)=1$, there exists a positive integer $c_i$ such that $x_1 \equiv c_i \mod p_i$, satisfying the equation $a_{i1} x_1  \equiv b_i \mod p_i$. By the Chinese Remainder Theorem, there exists a positive integer $c$ such that $c \equiv c_i \mod p_i$ for all $1 \le i \le m$. Thus, $x=c$ is a solution of the above simultaneous equations. 

Suppose that $n>1$ and the assertion holds for all $n=1, \dots, n-1$. We consider the following simultaneous (congruence) equation
\begin{align*}
\begin{cases}
a_{1n} x_n &\equiv b_1 \mod \gcd(a_{11}, \dots, a_{1\, n-1}, p_1) \\
a_{2n} x_n &\equiv b_2 \mod \gcd(a_{21}, \dots, a_{2\, n-1}, p_2) \\
&\vdots \\
a_{mn} x_n &\equiv b_m \mod \gcd(a_{m1}, \dots, a_{m\, n-1}, p_m)
\end{cases}
\end{align*}
A solution $x=c_n$ of the above exists for some positive integer $c_n$ by the induction hypothesis. Next, consider the following simultaneous (congruence) equation
\begin{align}\label{eq372}
\begin{cases}
a_{11} x_1 + a_{12} x_2 +\cdots +a_{1\, n-1} x_{n-1} &\equiv b_1 - a_{1n} c_n \mod p_1 \\
a_{21} x_1 + a_{22} x_2 +\cdots +a_{2\, n-1} x_{n-1} &\equiv b_2 - a_{2n} c_n \mod p_2 \\
&\vdots \\
a_{m1} x_1 + a_{m2} x_2 +\cdots +a_{m\, n-1} x_{n-1} &\equiv b_m - a_{mn} c_n \mod p_m
\end{cases}
\end{align}
Note that for all $1 \le i \le m$, we have 
\begin{align*}
\gcd(a_{i1}, \dots, a_{i\, n-1}, b_i - a_{in} c_n, p_i) 
= & \gcd(\gcd(a_{i1}, \dots, a_{i\, n-1}, p_i), b_i - a_{in} c_n)\\
= & \gcd(a_{i1}, \dots, a_{i\, n-1}, p_i),
\end{align*}
where the second equality follows since $\gcd(a_{i1}, \dots, a_{i\, n-1}, p_i)$ divides $b_i - a_{in} c_n$ by the definition of $c_n$. Set $g_i=\gcd(a_{i1}, \dots, a_{i\, n-1}, p_i)$ for $1\le i \le m$. We then observe that \eqref{eq372} is equivalent to 
\begin{align*}
\begin{cases}
(a_{11}/g_1) x_1 + (a_{12}/g_1) x_2 +\cdots +(a_{1\, n-1}/g_1) x_{n-1} &\equiv (b_1 - a_{1n} c_n)/g_1 \mod p_1/g_1 \\
(a_{21}/g_2) x_1 + (a_{22}/g_2) x_2 +\cdots +(a_{2\, n-1}/g_2) x_{n-1} &\equiv (b_2 - a_{2n} c_n)/g_2 \mod p_2/g_2 \\
&\vdots \\
(a_{m1}/g_m) x_1 + (a_{m2}/g_m) x_2 +\cdots +(a_{m\, n-1}/g_m) x_{n-1} &\equiv (b_m - a_{mn} c_n)/g_m \mod p_m/g_m
\end{cases}
\end{align*}
Then, by the induction hypothesis, there exist positive integers $c_1, \dots, c_{n-1}$ such that $x_1=c_1, \dots, x_{n-1}=c_{n-1}$ is a solution of the above simultaneous (congruence) equation. Therefore, we obtain that $x_1=c_1, \dots, x_{n-1}=c_{n-1}, x_n=c_n$ is a solution of \eqref{eqstar}.
\end{proof}

Now we can prove the main theorem of this article.

\begin{thm}\label{t38}
Suppose that $n_1, \dots, n_\ell$ are pairwise coprime and $G$ has no pseudo-reflection. 
%\begin{enumerate}[\rm(a)] 
%\item $n_1, \dots, n_\ell$ are pairwise coprime and 
%\item $G$ has no pseudo-reflection.
%\end{enumerate}
Then, for all characters $\calX$, $R^\calX$ is nonzero and $\tr_{R^G} (R^\calX) = R^\calX R^{\calX^{-1}}$ holds. 
\end{thm}

\begin{proof}
We prove the following claim.
\begin{claim}\label{claim1}
For all $1\le j \le d$, there exist positive integers $c_1, \dots, c_{j-1}, c_{j+1}, \dots, c_d$ such that $X_1^{c_1}\cdots X_{j-1}^{c_{j-1}} X_{j+1}^{c_{j+1}} \cdots X_{d}^{c_{d}} \in R^{(1,1,\dots, 1)}$.
\end{claim}

\begin{proof}[Proof of Claim \ref{claim1}]
By the symmetry, it is enough to prove the case where $j=d$. Since $G$ has no pseudo-reflection, $\gcd (t_{i1}, \dots, t_{i\, d-1}, n_i) = 1$ for all $1\le i \le \ell$ by Corollary \ref{c36}. Since we assume that $n_1, \dots, n_\ell$ are pairwise coprime, by Proposition \ref{p37}, there exist positive integers $c_1, \dots, c_{d-1}$ satisfying the equations: 
\begin{align*}
\begin{cases}
t_{11} c_1 + t_{12} c_2 +\cdots +t_{1\, d-1} c_{d-1} &\equiv 1 \mod n_1 \\
t_{21} c_1 + t_{22} c_2 +\cdots +t_{2\, d-1} c_{d-1} &\equiv 1 \mod n_2 \\
&\vdots \\
t_{\ell 1} c_1 + t_{\ell 2} c_2 +\cdots +t_{\ell \, d-1} c_{d-1} &\equiv 1 \mod n_\ell
\end{cases}
\end{align*}
In other words, $\sigma_i(X_1^{c_1}X_2^{c_2} \cdots X_{d-1}^{c_{d-1}}) = \xi_i X_1^{c_1}X_2^{c_2} \cdots X_{d-1}^{c_{d-1}}$ for all $1\le i \le \ell$. This proves that 
\[
X_1^{c_1}X_2^{c_2} \cdots X_{d-1}^{c_{d-1}} \in R^{(1,1,\dots, 1)}
\]
as desired.
\end{proof}

By Claim \ref{claim1}, for all positive integers $p$ and all $1 \le j \le d$, $(X_1^{c_1}\cdots X_{j-1}^{c_{j-1}} X_{j+1}^{c_{j+1}} \cdots X_{d}^{c_{d}})^p \in R^{(p,p,\dots, p)}$. Since $\gcd (\{(X_1^{c_1}\cdots X_{j-1}^{c_{j-1}} X_{j+1}^{c_{j+1}} \cdots X_{d}^{c_{d}})^p \mid 1 \le j \le d \}) =1$, it follows that $\gcd (R^{(p,p,\dots, p)}) =1$. 
%Hence, $X_j$ does not divide $\gcd (R^{(p,p,\dots, p)})$ for all $1 \le j \le d$, that is, $\gcd (R^{(p,p,\dots, p)}) =1$. 
On the other hand, for each character $\calX$, $R^\calX = R^{(p,p,\dots, p)}$ for some $1 \le p \le n_1n_2\cdots n_\ell$ by the Chinese Remainder Theorem. Therefore, we have $\gcd (R^\calX) =1$ for all characters $\calX$. Thus, the assertion follows by Lemma \ref{l32} and Proposition \ref{p33}.
\end{proof}

\begin{cor}\label{c39}
Suppose that $G$ is cyclic and has no pseudo-reflection. Then, for all characters $\calX$, $R^\calX$ is nonzero and $\tr_{R^G} (R^\calX) = R^\calX R^{\calX^{-1}}$ holds. 
\end{cor}

The following example shows that the equation $\tr_{R^G} (R^\calX) = R^\calX R^{\calX^{-1}}$ does not hold if we remove the assumption that $n_1, \dots, n_\ell$ are pairwise coprime in Theorem \ref{t38}.

\begin{ex} \label{traceneq}
Let $\xi_1$ and $\xi_2$ be the $4$th primitive root of $1\in K$ and the $6$th primitive root of $1\in K$, respectively. Suppose that $G=\langle \sigma_1, \sigma_2 \rangle$, where 
$\sigma_1$ and $\sigma_2$ are $3\times 3$ matrix diagonalizing with $\xi_1, \xi_1, \xi_1$ and $\xi_2, \xi_2^2, \xi_2^3$. Then, the following hold true. 
\begin{enumerate}[\rm(i)] 
\item $R^\calX$ is nonzero for each character $\calX$.
\item $R^{(1,0)}$ is a canonical $R^G$-module and $\tr_{R^G} (R^{(1,0)}) \supsetneq R^{(1,0)} R^{(3,0)}$.
\end{enumerate}
\end{ex}

\begin{proof}
(i): It is straightforward to check that $X_1 X_2 X_3^{23} \in R^{(1,0)}$ and $X_2^{23} X_3 \in R^{(0,1)}$. Hence, $(X_1 X_2 X_3^{23})^s (X_2^{23} X_3)^t \in R^{(s,t)}$ for all non-negative integers $s, t$. It follows that for each character $\calX$, the semi-invariants of weight $\calX$ are nonzero. 

(ii): $R^{(1,0)}$ is a canonical $R^G$-module by \cite[Theorem 6.4.2(b)]{BH}. We have $\tr_{R^G} (R^{(1,0)}) \supseteq R^{(1,0)} R^{(3,0)}$ by Lemma \ref{l32} and Proposition \ref{p33}. Thus, we complete the proof by showing the following claim.
\begin{claim}\label{claim2}
The following hold true.
\begin{enumerate}[\rm(1)] 
\item $X_2$ divides all monomials in $R^{(1,0)}$. 
\item $X_2$ divides all monomials in $R^{(1,0)}R^{(3,0)}$.
\item There exists a monomial $f$ in $\tr_{R^G} (R^{(1,0)})$ such that $X_2$ does not divide $f$.
%\item $X_2$ divides $\gcd (R^{(1,0)})$. 
%\item $X_2$ divides $\gcd (R^{(1,0)}R^{(3,0)})$.
%\item $X_2$ does not divide $\gcd (\tr_{R^G} (R^{(1,0)}) )$.
\end{enumerate}
\end{claim}

\begin{proof}[Proof of Claim \ref{claim2}]
(1): Suppose the contrary. Then, there exist positive integers $a, b$ such that $X_1^aX_3^b\in R^{(1,0)}$. This is equivalent to saying that 
%Since $R^{(1,0)}$ is generated by monomials (Remark \ref{rem0}(2)), there exist positive integers $a, b$ such that $X_1^aX_3^b\in R^{(1,0)}$. This is equivalent to saying that 
\begin{align*}
\begin{cases}
a+b &\equiv 1 \mod 4\\
a+3b&\equiv 0 \mod 6
\end{cases}
\end{align*}
This implies that $2$ divides $(a+3b)-(a+b-1) = 2b+1$, which is a contradiction. Hence, $X_2$ divides all monomials in $R^{(1,0)}$. 

(2): This follows from Claim \ref{claim2}(1). 

(3): Note that $\frac{X_{1}^{11}X_3}{X_2} \in S^{-1} R$, where $S=R^G\setminus \{0\}$ is a multiplicative closed subset of $R^G$, and set $\alpha = \frac{X_{1}^{11}X_3}{X_2}$. Then $\alpha R^{(1,0)} \subseteq R$ by Claim \ref{claim2}(1). 
It is straightforward to check that 
\[
\sigma_1(\alpha) = \sigma_1(X_{1}^{11}X_3)/\sigma_1(X_{1}^{11}X_3) = \xi_1^3 \alpha \quad \text{and} \quad \sigma_2(\alpha) = \sigma_2(X_{1}^{11}X_3)/\sigma_2(X_{1}^{11}X_3) = \alpha.
\]
It follows that $\alpha \in R^G:_{S^{-1} R} R^{(1,0)}$. On the other hand, one can check that $X_1 X_2 X_3^{23} \in R^{(1,0)}$. Therefore, we obtain that 
\[
X_1^{12} X_3^{24} = \alpha X_1 X_2 X_3^{23} \in  (R^G:_{S^{-1}R} R^{(1,0)}) R^{(1,0)} = \tr_{R^G}(R^{(1,0)})
\]
by Lemma \ref{l32}. Hence, we get $f=X_1^{12} X_3^{24}$ as desired. 
%This shows that $X_2$ does not divide $\gcd (\tr_{R^G} (R^{(1,0)}) )$.
\end{proof}

By Claim \ref{claim2}(2) and (3), $\tr_{R^G} (R^{(1,0)}) \ne R^{(1,0)} R^{(3,0)}$; hence, we conclude the latter assertion in Example~\ref{traceneq}(ii).
\end{proof}

The following provides a criterion for semi-invariants to be locally free on the graded punctured spectrum. We say that for $V\subseteq \Spec R^{G}$, $R^\calX$ is {\bf locally free on $V$} if $R^\calX_\fkp$ is $R^G_\fkp$-free for all $\fkp\in V$. 

%The following is a criterion for semi-invariants to be locally free on the graded punctured spectrum. 

\begin{thm}\label{t311}
Let $n=\prod_{i=1}^\ell n_i$. 
%Let $\fkm_G$ denote the graded maximal ideal of $R^G$. 
For each character $\calX$, consider the following conditions.
\begin{enumerate}[\rm(1)] 
\item $R^\calX$ is locally free on ${}^*\Spec R^G\setminus \{\fkm_G\}$.
%$R^\calX_\fkp$ is $R^G_\fkp$-free for all $\fkp\in {}^*\Spec R^G\setminus \{\fkm_G\}$.
\item $(X_1^n, \dots, X_d^n) \subseteq \tr_{R^G} (R^\calX)$.
\item For all $1 \le j \le d$, there exists $0<u_j\le n$ such that $X_j^{u_j}\in R^\calX$.
\end{enumerate}
Then {\rm (3) $\Rightarrow$ (2) $\Rightarrow$ (1)} holds. {\rm (1) $\Rightarrow$ (3)} also holds if $n_1, \dots, n_\ell$ are pairwise coprime and $G$ has no pseudo-reflection. 
\end{thm}

\begin{proof}
(3) $\Rightarrow$ (2): Note that $X_j^n\in R^G$ for all $1 \le j \le d$. Therefore, since $X_j^{u_j}\in R^\calX$, we have $X_j^{n-u_j} \in R^{\calX^{-1}}$. Hence, by Lemma \ref{l32} and Proposition \ref{p33}, we observe $X_j^n = X_j^{n-u_j}X_j^{u_j} \in \tr_{R^G} (R^\calX)$ for all $1 \le j \le d.$

(2) $\Rightarrow$ (1): Since $(X_1^n, \dots, X_d^n)$ is an $\fkm_G$-primary ideal of $R^G$, the assertion (2) implies that $R^\calX_\fkp$ has an $R^G_\fkp$-free summand for all $\fkp\in {}^*\Spec R^G\setminus \{\fkm_G\}$ (Remark \ref{f22}(4)). Since $R^\calX$ is a torsion-free $R^G$-module of rank $1$ (Remark \ref{rem0}(5)), it follows that $R^\calX_\fkp$ is an $R^G_\fkp$-free module (of rank $1$).

(1) $\Rightarrow$ (3): By the assumption (1), $\tr_{R^G} (R^\calX)$ is an $\fkm_G$-primary ideal of $R^G$ (Remark \ref{f22}(4)). Hence, for all $1\le j \le d$, there exists a positive integer $v_j$ such that $X_j^{v_j}\in \tr_{R^G} (R^\calX)$. By Theorem \ref{t38}, it follows that $X_j^{v_j}\in R^\calX R^{\calX^{-1}}$. Since $R^\calX \subseteq R$ and $R^{\calX^{-1}} \subseteq R$, there exists $0<u_j \le v_j$ such that $X_j^{u_j} \in R^\calX$. Since $X_j^n\in R^G$, by considering $u_j$ modulo $n$, we can replace $u_j$ to satisfy $0 < u_j \le n$. 
\end{proof}

\begin{rem}
The condition (3) in Theorem \ref{t311} can be checked by a simple calculation. Indeed, letting $R^\calX = R^{(s_1, \dots, s_\ell)}$, the condition (3) in Theorem \ref{t311} is equivalent to stating that for all $1 \le j \le d$, there exists $0 < u_j \le n$ satisfying the following simultaneous (congruence) equation:
\begin{align*}
\begin{cases}
u_j t_{1j} &\equiv s_1 \pmod{n_1}\\
u_j t_{2j} &\equiv s_2 \pmod{n_2}\\
&\vdots\\
u_j t_{\ell j} &\equiv s_\ell \pmod{n_\ell}
\end{cases}
\end{align*}
\end{rem}

\begin{cor}\label{c313}
Suppose that $n_1, \dots, n_\ell$ are pairwise coprime, and $G$ has no pseudo-reflection. Set $n = \prod_{i=1}^\ell n_i$. 
Then the following are equivalent.
\begin{enumerate}[\rm(1)] 
\item For each character $\calX$, $R^\calX$ is locally free on ${}^*\Spec R^G\setminus \{\fkm_G\}$. 
\item For all $1 \le j \le d$, $X_j^{1}, X_j^2, \dots, X_j^{n}$ are in different semi-invariants. 
%\item $R^{(1,1,\dots, 1)}$ is locally free for all $\fkp\in {}^*\Spec R^G\setminus \{\fkm_G\}$. 
\end{enumerate}
\end{cor}

\begin{proof}
Note that the number of all characters is $n$ (see before Remark \ref{rem0}). Thus, for all $1 \le j \le d$, $X_j^{1}, X_j^2, \dots, X_j^{n}$ are in different semi-invariants if and only if for each character $\calX$, there exists $0 < u_j \le n$ such that $X_j^{u_j}\in R^\calX$. The latter is equivalent to stating that all semi-invariants $R^\calX$ are locally free on ${}^*\Spec R^G\setminus \{\fkm_G\}$ by Theorem \ref{t311}. 
\end{proof}

\begin{cor}\label{c314}
Suppose that $G$ is cyclic (i.e., $\ell=1$) and has no pseudo-reflection. %Set $n=\prod_{i=1}^\ell n_i$. 
Then the following are equivalent.
\begin{enumerate}[\rm(1)] 
\item For each character $\calX$, $R^\calX$ is locally free on ${}^*\Spec R^G\setminus \{\fkm_G\}$. 
\item For all $1 \le j \le d$, $\gcd(t_{1j}, n_1) =1$. 
%\item $R^{(1,1,\dots, 1)}$ is locally free for all $\fkp\in {}^*\Spec R^G\setminus \{\fkm_G\}$. 
\end{enumerate}
\end{cor}

\begin{proof}
Since $G$ is cyclic, the condition of Corollary \ref{c313}(2) is equivalent to stating that for all $1 \le j \le d$, the integers $t_{1j}, 2t_{1j},\dots, n_1t_{1j}$ are different modulo $n_1$. This is equivalent to the condition (2) of this assertion. 
\end{proof}

%%%%%%%%%%%%%%%%%%%%%%%%%%%%%%%%%%%%%%%%%%%%%%%%%%%%%%%%%%%%
\section{Proximity to Gorenstein Properties}\label{section4}

%\section{Properties close to being Gorenstein}\label{section4}

In this section, we explore properties that are close to being Gorenstein under Setup \ref{setup1} by computing the trace of the canonical module.

We first apply Theorem \ref{t311} to examine the graded non-Gorenstein locus of $R^G$. 
In a more general context, let $A$ be a Cohen-Macaulay ring with unique graded maximal ideal. Assuming the existence of a graded canonical module $\omega_A$ for $A$, it is established that $\tr_A(\omega_A)$ defines the \textbf{graded non-Gorenstein locus}, denoted by
\[
\{ \fkp\in {}^*\Spec A \mid \fkp\supseteq \tr_A(\omega_A) \} = \{ \fkp\in {}^*\Spec A \mid A_\fkp \text{ is not Gorenstein}, \}
\]
where ${}^*\Spec A$ denotes the set of graded prime ideals of $A$ (cf. Remark \ref{f22}(4)). We say that $A$ is {\bf Gorenstein on $V$} for $V\subseteq \Spec A$ if $A_\fkp$ is Gorenstein for all $\fkp\in V$. 
%(for example, there is a proof in \cite[Lemma 2.1]{HHS} for a local case. The graded local case can be also proved by a similar way.). 

On the other hand, for rings $R^G$ of invariants, it is known that the inverse determinant character describes a canonical $R^G$-module. Here, the group homomorphism
\[
\mathrm{det}^{-1}: G \to \mathrm{GL}(K); \sigma \mapsto \det(\sigma)^{-1}
\] 
is referred as the {\bf inverse determinant character}. We also define $\det: G \to \mathrm{GL}(K); \sigma \mapsto \det(\sigma)$ for convenience. By \cite[Theorem 6.4.2(b)]{BH}, it is known that $\omega_{R^G} \cong R^{\det^{-1}}(-d)$ as graded $R^G$-modules. 
Therefore, by applying Theorem \ref{t311} with $\calX=\det^{-1}$ and $\calX=\det$, we obtain the following. (Note that $\tr_{R^G} (R^\calX) = R^\calX R^{\calX^{-1}} = \tr_{R^G} (R^{\calX^{-1}})$ under the assumption of Theorem \ref{t38}.)

\begin{cor}\label{c41}
    Suppose that $n_1, \dots, n_\ell$ are pairwise coprime, and $G$ has no pseudo-reflection. Set $n=\prod_{i=1}^\ell n_i$. Then the following are equivalent.
    \begin{enumerate}[\rm(1)] 
        \item $R^G$ is Gorenstein on ${}^*\Spec R^G\setminus \{\fkm_G\}$.
        \item For all $1 \le j \le d$, there exists $0<u_j\le n$ such that $X_j^{u_j}\in R^{\det^{-1}}$.
        \item For all $1 \le j \le d$, there exists $0<u_j\le n$ such that $X_j^{u_j}\in R^{\det}$.
    \end{enumerate}
\end{cor}

We further consider the nearly Gorenstein property of $R^G$. Below we recall the definition of nearly Gorenstein rings.

\begin{defn} (\cite[Definition 2.2]{HHS})
    Let $A$ be a Cohen-Macaulay local ring or a positively graded $K$-algebra over a field $K$. Set $\fkm_A$ as the maximal ideal of $A$ or the graded maximal ideal of $A$. Suppose that $A$ admits a canonical module $\omega_A$. Then $A$ is called \textbf{nearly Gorenstein} if $\tr_A(\omega_A) \supseteq \fkm_A$.
\end{defn}

By Theorem \ref{t38}, we immediately get the following.

\begin{cor}\label{c43}
    Suppose that $n_1, \dots, n_\ell$ are pairwise coprime and $G$ has no pseudo-reflection. 
    %\begin{enumerate}[\rm(a)] 
    %\item $n_1, \dots, n_\ell$ are pairwise coprime and 
    %\item $G$ has no pseudo-reflection.
    %\end{enumerate}
    Then the following are equivalent.
    \begin{enumerate}[\rm(1)] 
        \item $R^G$ is nearly Gorenstein.
        \item $R^{\det} R^{\det^{-1}}\supseteq \fkm_G$.
        \item For each monomial $f$ generating $\fkm_G$, there exists a monomial $g\in R^{\det}$ such that $g$ divides $f$.
    \end{enumerate}
    %$R^G$ is nearly Gorenstein if and only if $R^{\det} R^{\det^{-1}}\supseteq \fkm_G$.
\end{cor}

\begin{proof}
    (1)$\Leftrightarrow$(2): This follows from Theorem \ref{t38}. 

    (2)$\Leftrightarrow$(3): Since $R^{\det}$ and $R^{\det^{-1}}$ are generated by monomials (Remark \ref{rem0}(2)), so is $R^{\det} R^{\det^{-1}}$. Since $\fkm_G$ is also generated by monomials, we get the assertion.
\end{proof}

We conclude this article with several examples. 
Caminata and Strazzanti \cite[Corollary 2.5]{CS} have proven that $R^G$ is nearly Gorenstein if $d=2$ and $\ell=1$ (hence, $R^G_\fkp$ is Gorenstein for all $\fkp\in {}^*\Spec R^G\setminus \{\fkm_G\}$). However, such an assertion cannot be expected even in the case where $d=3$ and $\ell=1$. 

% Thus, we provide several examples in the case where $d=3$ and $\ell=1$. 

\begin{ex} 
    Suppose that $d=3$ and $\ell=1$. Set $n=n_1$. Then the following hold true. 
    \begin{enumerate}[\rm(1)] 
        \item Let $n=4$ and $(t_{11}, t_{12}, t_{13}) = (1,1,3)$. Then $R^G$ is nearly Gorenstein but not Gorenstein. 
        \item Let $n=4$ and $(t_{11}, t_{12}, t_{13})  = (1,2,3)$. Then $R^G$ is Gorenstein on ${}^*\Spec R^G\setminus \{\fkm_G\}$, but $R^G$ is not nearly Gorenstein. 
        \item Let $n=6$ and $(t_{11}, t_{12}, t_{13})  = (1,1,3)$. Then $R^G$ is not Gorenstein on ${}^*\Spec R^G\setminus \{\fkm_G\}$. 
    \end{enumerate}
\end{ex}

\begin{proof}
    (1): This result is recorded in \cite[Table 1]{CS}, but we note a way to check the nearly Gorenstein property of $R^G$ for the convenience of readers.  
    We have $R^{\det^{-1}} =R^{(3)}$ since $\det^{-1}: G \to \mathrm{GL}(K); \sigma_1 \mapsto \det(\sigma_1)^{-1}=(\xi_1^{1+1+3})^{-1}=\xi_1^{-5} =\xi_1^3$. Since $R^{(3)} \subseteq R$, $R^{(1)} \subseteq R$, and $1\not\in R^{(3)}$, we get $1 \not\in R^{(3)} R^{(1)} = \tr_{R^G} (R^{(3)})$ by Theorem \ref{t38}. It follows that $R^{\det^{-1}} =R^{(3)}\not \cong R^G$ (see Remark \ref{f22}(4)). By \cite[Theorem 6.4.2(b)]{BH}, $R^G$ is not Gorenstein. (If one assumes that $K$ is a field of characteristic $0$, then this follows from \cite[Theorem 6.4.10]{BH}.)

    We next prove that $R^G$ is nearly Gorenstein. By Macaulay2 (\cite{Mac2}), one can check that the graded maximal ideal $\fkm_G$ of $R^G$ is 
    \[
    (X_1^3X_2, X_1^4, X_1^2X_2^2, X_1X_2^4, X_2^4, X_3^4, X_1X_3, X_2X_3).
    \]
    On the other hand, we have $R^{\det}=R^{(1)}$. Thus, one can also check that $X_1, X_2, X_3^3\in R^{(1)} = R^{\det}$. Since all the above monomials generating $\fkm_G$ are divided by some of  $X_1, X_2, X_3^3\in R^{\det}$, we get the assertion by Corollary \ref{c43}.

    (2): By Macaulay2 (\cite{Mac2}), one can check that the graded maximal ideal $\fkm_G$ of $R^G$ is 
    \[
    (X_2^2, X_1X_3, X_2X_3^2, X_1^2X_2, X_3^4, X_1^4).
    \]
    On the other hand, one can also check that $R^{\det} = R^{(2)}$. Then, both $X_1$ and $X_3$ are not in $R^{\det}$. Hence, the monomial $X_1X_3$, a part of monimal generators of $\fkm_G$, is not divided by any monomial in $R^{\det}$. By Corollary \ref{c43}, $R^G$ is not nearly Gorenstein. 

    However, one can also check that $X_1^2, X_2, X_3^2\in R^{(2)} = R^{\det}$; hence, $R^G$ is Gorenstein on ${}^*\Spec R^G\setminus~\{\fkm_G\}$ by Corollary \ref{c41}.

    (3): We have $R^{\det^{-1}} =R^{(1)}$. Then $X_3, \dots, X_3^6\not\in R^{\det^{-1}}$. By Corollary \ref{c41},  $R^G$ is not Gorenstein on ${}^*\Spec R^G\setminus \{\fkm_G\}$. 
\end{proof}

As a known result, if $R^G$ is a Veronese subring of $R$, then $\tr_{R^G} (R^\calX)\supseteq \fkm_G$ for all characters $\calX$ (\cite[Theorem 4.6]{HHS}). In particular, all Veronese subrings of $R$ are nearly Gorenstein. The following example shows that the nearly Gorenstein property of $R^G$ does not imply $\tr_{R^G} (R^\calX)\supseteq \fkm_G$ for all characters $\calX$ in general. 

\begin{ex} 
    Suppose that $d=3$ and $\ell=1$. Set $n=n_1$. Let $n=6$ and $(t_{11}, t_{12}, t_{13}) = (1,1,2)$. Then $R^G$ is nearly Gorenstein, but  $R^{(1)}$ is not locally free on ${}^*\Spec R^G\setminus \{\fkm_G\}$ (and thus $\tr_{R^G} (R^{(1)})\subsetneq \fkm_G$).
\end{ex}

\begin{proof}
    One can check that $R^G$ is nearly Gorenstein (see also \cite[Table 1]{CS}). On the other hand, $X_3, X_3^2, \dots, X_3^6 \not\in R^{(1)}$ since $2u\not\equiv 1 \mod 6$ for all $1\le u \le 6$. Hence, $R^{(1)}$ is not locally free on ${}^*\Spec R^G\setminus \{\fkm_G\}$ by Theorem \ref{t311} (see also Corollary \ref{c314}).
\end{proof}

%Let us give one example which is non-cyclic. 

%\begin{ex} \begin{color}{red} NON-cyclic example!\end{color}
%\end{ex}

%\begin{proof}

%\end{proof}

%%%%%%%%%%%%%%%%%%%%%%%%%%%%%%%%%%%%%%%%%%%%%%%%%%%%%%%%%%%%

%\addcontentsline{toc}{section}{references}

\end{document}